\documentclass[reqno]{amsart}

\pdfoutput=1

\usepackage{hyperref}

\usepackage[french, english]{babel}
\usepackage[alphabetic]{amsrefs}
\usepackage[pdftex]{color}
\usepackage{amsxtra}
\usepackage{upref}
\usepackage{epstopdf}
\usepackage{graphicx,color}
\usepackage{hyperref}
\usepackage{longtable}
\usepackage{graphicx}
\usepackage{amsmath, amsfonts, amssymb,amscd, amstext, amsthm, 
mathtools, bezier, mathrsfs,enumerate}
\usepackage{url}
\usepackage{float}
\usepackage{caption}
\usepackage{tikz}
\usetikzlibrary{intersections,positioning,patterns,arrows,decorations.markings}
\usepackage[fit]{truncate}
\usepackage{anyfontsize}
\usepackage{t1enc}
\usepackage{synttree}
\usepackage{pst-plot}
\psset{algebraic}

\DeclareFontFamily{OML}{rsfs}{\skewchar\font'177}
\DeclareFontShape{OML}{rsfs}{m}{n}{ <5> <6> rsfs5 <7> <8> <9> rsfs7
  <10> <10.95> <12> <14.4> <17.28> <20.74> <24.88> rsfs10 }{}
\DeclareMathAlphabet{\mathfs}{OML}{rsfs}{m}{n}

\newtheorem{theorem}{Theorem}

\newtheorem{proposition}[theorem]{Proposition}

\theoremstyle{definition}
\newtheorem{definition}[theorem]{Definition}

\newtheorem{example}{Example}

\numberwithin{equation}{section}
\numberwithin{theorem}{section}

\newcommand{\intav}[1]{\mathchoice {\mathop{\vrule width 6pt height 3 pt depth  -2.5pt
\kern -8pt \intop}\nolimits_{\kern -6pt#1}} {\mathop{\vrule width
5pt height 3  pt depth -2.6pt \kern -6pt \intop}\nolimits_{#1}}
{\mathop{\vrule width 5pt height 3 pt depth -2.6pt \kern -6pt
\intop}\nolimits_{#1}} {\mathop{\vrule width 5pt height 3 pt depth
-2.6pt \kern -6pt \intop}\nolimits_{#1}}}

\newcommand{\intavl}[1]{\mathchoice {\mathop{\vrule width 6pt height 3 pt depth  -2.5pt
\kern -8pt \intop}\limits_{\kern -6pt#1}} {\mathop{\vrule width 5pt
height 3  pt depth -2.6pt \kern -6pt \intop}\nolimits_{#1}}
{\mathop{\vrule width 5pt height 3 pt depth -2.6pt \kern -6pt
\intop}\nolimits_{#1}} {\mathop{\vrule width 5pt height 3 pt depth
-2.6pt \kern -6pt \intop}\nolimits_{#1}}}

\newcommand{\cc}{\mathscr{C}}
\newcommand{\aaa}{\mathscr{A}}
\newcommand{\F}{\mathcal{F}}
\newcommand{\G}{\mathcal{G}}
\newcommand{\V}{\mathbb{V}}

\newcommand{\N}{\mathbb{N}}

\newcommand{\pp}{\mathscr{P}}
\newcommand{\etal}{\textit{et al\ }}
\newcommand{\eg}{e.g.\ }

\usepackage{scrextend}
\deffootnotemark{\textsuperscript{(\thefootnotemark)}}

\begin{document}


\title[Topological conjugacy]{Topological conjugacy for unimodal nonautonomous discrete dynamical systems}

\author{Ermerson Araujo}
\date{\today}

\address{Ermerson Araujo, Departamento de Matem\'atica, Universidade Federal do Cear\'a (UFC), 
Campus do Pici, Bloco 914, CEP 60440-900, Fortaleza -- CE, Brasil}
\email{ermersonaraujo@gmail.com}

\begin{abstract}
The goal of this article is to study how combinatorial equivalence implies 
topological conjugacy. For that, we introduce the concept of kneading sequences for nonautonomous 
discrete dynamical systems and show that these sequences are a complete invariant for
topological conjugacy classes.

\end{abstract}

\maketitle


\section{Introduction}

A {\it nonautonomous discrete dynamical system} (short NDS) is a pair $(X,\mathcal{F})$, 
where $X$ is a metric space and
$\mathcal{F}=(f_n)_{n\geq1}$ is a sequence of continuous maps 
$f_n:X\to X$. The orbits of the system are described by the
maps $f^\ell_n:X\to X$, defined by
\[
f^\ell_n(x):=(f_{n+\ell-1}\circ\cdots\circ f_n)(x)\; \text{for each}\; n,\ell\in\N \;\text{and}\; x\in X,
\]
\[
f_n^0:=\text{id}_{X}\; \text{for each}\; n\in\N.
\]
The classical autonomous setting is obtained by letting $f_n=f$, for every $n\geq1$.
Furthemore, we define $f_n^{-\ell}:=(f^\ell_n)^{-1}$, which is only applied to sets. (We do not
assume that the maps $f_n$ are invertible.) 

Nonautonomous discrete dynamical systems were introduced by S. Kolyada and 
L. Snoha \cite{ks} motivated by the desire to understand better the 
topological entropy of skew products.
In recent years, a large number of papers have been devoted to dynamical
properties in nonautonomous discrete systems. 
Huang \etal \cite{hwzz} introduced and studied 
topological pressure for nonautonomous discrete dynamical systems.
Metric entropy of NDS has been studied in \cite{kawan11} and \cite{kawan22}.
The notion of chaos was extended to NDS setting by many authors 
(\eg \cite{tichen,shi,wuzhu,zhshsh}).
So, although recognizably distinct from classical autonomous dynamical systems, the theory
of the nonautonomuos discrete dynamical systems
has developed into a highly active field of research.
In this way, it is natural to search for ways to classify the NDS
in classes with similar dynamical behavior. 

The easiest way to see when two NDS have the same dynamical behavior 
is when there exists a topological conjugacy between them. 
For example, topological entropy for NDS is invariant by topological conjugacy \cite[Sec. 5]{ks}.
Let $(X, \mathcal{F})=(X, (f_n)_{n\geq1})$ and $(Y, \mathcal{G})=(Y, (g_n)_{n\geq1})$
be two nonautonomous discrete dynamical systems. We say that 
$(X, \mathcal{F})$ and $(Y, \mathcal{G})$
are topologically conjugate if there exists a sequence 
$(h_n)_{n\geq1}$ of homeomorphisms from $X$ into $Y$ such that
both families $(h_n)_{n\geq1}$ and $(h_n^{-1})_{n\geq1}$ are equicontinuous and
$h_{n+1}\circ f_n=g_n\circ h_n$ for every $n\geq1$.
If $h_n$ is a continuous surjective map for all $n\geq1$, then we say that
$(X, \mathcal{F})$ and $(Y, \mathcal{G})$
are topologically semi-conjugate. When $X, Y$ are intervals we also require
that all homeomorphisms $h_n$ are order preserving. 

We can not remove the equicontinuity condition of the family $(h_n)_{n\geq1}$ 
as this would imply that all NDS $(f_n)_{n\geq1}$, with $f_n$ homeomorphism
for each $n\geq1$, are topologically conjugate to a trivial NDS, see \cite[Prop. 2.1]{AF}. 

Thus, inspired by this we ask the following.

\medskip
\noindent
{\bf Problem 1:} Let $(X, \mathcal{F})$ and $(Y, \mathcal{G})$ be two 
nonautonomous discrete dynamical systems. 
Under which conditions $(X, \mathcal{F})$ and $(Y, \mathcal{G})$
are topologically (semi-)conjugate?

\medskip
In this short work, we introduce the notion of combinatorial equivalence for NDS 
on the particular case where $X$ is an interval and $f_n$ is a unimodal 
map for every $n\geq1$. We can give the following answer to Problem 1.

\begin{theorem}\label{thmA}
Let $(J^\F, \F)$ and $(J^\G, \G)$ be two unimodal nonautonomous
discrete dynamical systems and assume that
both satisfy the limit property. Then $(J^\F,\F)$ and $(J^\G,\G)$ are topologically 
conjugate if and only if they have the same kneading sequence.
\end{theorem}

To prove this theorem, we will construct kneading sequences in a similar way as
Milnor and Thurston made in their famous paper \cite{MT}.

%


\section{Kneading sequences for UNDS}

Let $(J^\mathcal{F}, \mathcal{F})$ be the NDS defined as follows:
let $J^\F=[a^\F, b^\F]$ be an interval, and 
$f_n:J^\F\to J^\F$ be a continuous map satisfying
$f_n(a^\F)=f_n(b^\F)=a^\F$ for each $n\geq1$. 
Besides that, there is $c_n^\F\in(a^\F,b^\F)$ 
such that $f_n\restriction_{[a^\F, c_n^\F]}$ is strictly increasing 
and $f_n\restriction_{[c_n^\F, b^\F]}$ is strictly decreasing.
The points $\{c_n^\F:\; n\geq1\}$ are called {\it turning points} of the NDS.
We call the NDS defined above {\it unimodal nonautonomous discrete dynamical system} (short UNDS)

Let's proceed by constructing the symbolic space for UNDS.
Consider the alphabet $\mathcal{A}_\F=\{L, c_n^\F, R: n\geq1\}$.  

\medskip
\noindent
{\sc Address of a point:} Let $n\geq1$. The {\em address} of a point $x\in J^\F$ 
on the level $n$ is the letter $i_{\F,n}(x)\in\mathcal A_\F$ defined by
$$
i_{\F,n}(x)= \left\{
\begin{array}{ll}
L    & \textrm{, if }  x\in[a^\F,c_n^\F)     \\
c_n^\F  & \textrm{, if }  x=c_n^\F \\
R    & \textrm{, if }  x\in(c_n^\F,b^\F].
\end{array}
\right.
$$

\medskip
\noindent
{\sc Itinerary of a point:} Let $n\geq1$. The {\em itinerary} of a point $x\in J^\F$
on the level $n$ is the sequence $I_{\mathcal{F},n}(x)\in\mathcal A_\F^{\{0,1,2,\ldots\}}$ defined by 
\[
I_{\mathcal{F},n}(x)=(i_{\mathcal{F},n}(x),i_{\mathcal{F},n+1}(f_n^1(x)),\ldots,
i_{\F,n+\ell}(f_n^\ell(x)),\ldots).
\]

\medskip
\noindent
{\sc Kneading Sequence:} The {\em kneading sequence} of $(J^\F,\F)$ 
is the sequence $\V(\F)=\{\V_n^\F\}_{n\geq1}$, where $\V_n^\F:=I_{\F,n}(c_n^\F)$.

\medskip  
The proposition below ensures that the kneading sequences 
are preserved by topological conjugacy. 

\begin{proposition}\label{prop1}
Let $(J^\F,\F)$ and $(J^\mathcal{G},\mathcal{G})$ be two UNDS. 
If $(J^\F,\F)$ and $(J^\mathcal{G},\mathcal{G})$ are topologically conjugate
then $\V(\F)=\V(\mathcal{G})$.
\end{proposition}

Before going to the proof, we observe that $\V(\F)=\V(\mathcal{G})$ means
that we are identifying $\mathcal{A}_\F$ with $\mathcal{A}_\G$.

\begin{proof}
We have shown that $I_{\F,n}(c_n^\F)=I_{\G,n}(c_n^\G)$ for all $n\geq1$.
Let $(h_n)_{n\geq1}$ be a conjugacy between $(J^\F,\F)$ and $(J^\G,\G)$.
Since $h_n:J^\F\to J^\G$ is an order preserving homeomorphism,
$h_n$ sends the increasing (decreasing) interval of $f_n$ to the increasing (decreasing)
interval of $g_n$. Thus, for all $n\geq1$ we have $h_n(c_n^\F)=c_n^\G$ and 
\[
(h_{n+\ell}\circ f_n^\ell)(c_n^\F)=(g_n^\ell\circ h_n)(c_n^\F)=g_n^\ell(c_n^\G),\ \forall\ell\geq 1.
\]
Once $h_{n+\ell}$ is ordering preserving on $J^\F$ for all $\ell\geq1$,
\[
i_{\F,n+\ell}(f_n^\ell(c_n^\F))=i_{\G,n+\ell}(g_n^\ell(c_n^\G)).
\]
Therefore $I_{\F,n}(c_n^\F)=I_{\G,n}(c_n^\G)$. 
\end{proof}


\section{Combinatorial equivalence for UNDS}

The purpose of this section is to introduce the concept of
combinatorial equivalence between two UNDS.
We will prove that equality between kneading sequences
is sufficient to ensure combinatorial equivalence.

For each $n\geq1$ set
\[
\cc_n(\F)=\left\{x\in J^\F;\;\exists\; \ell\geq0 \;
\textrm{such that}\; f_n^\ell(x)=c_{n+\ell}^\F\right\}.
\]

\begin{definition}
We say that two UNDS $(J^\F,\F)$ and $(J^\mathcal{G},\mathcal{G})$ with
turning points $\{c^\F_n;\;n\geq1\}$ respectively $\{c^\mathcal{G}_n;\;n\geq1\}$ 
are {\it combinatorially equivalent} if there exists a family of order preserving bijections
$h_n:\cc_n(\F)\to\cc_n(\mathcal{G})$ such that $h_{n+1}\circ f_n=g_n\circ h_n$
on $\cc_n(\F)\backslash\{c_n^\F\}$ for all $n\geq1$.
\end{definition}

In the next section we prove that, with the right hypotheses (limit property), the 
equality between kneading sequences implies topological conjugacy. 
The following theorem is fundamental for that.

\begin{theorem}\label{thm1}
Let $(J^\F,\F)$ and $(J^\mathcal{G},\mathcal{G})$ be two UNDS with
kneading sequences $\V(\F)$ and $\V(\mathcal{G})$. If
$\V(\F)=\V(\G)$, then $(J^\F,\F)$ and $(J^\G,\mathcal{G})$
are combinatorially equivalent. 
\end{theorem}

The proof of Theorem \ref{thm1} is inspired by \cite[Thm. 1]{Rand}.
Before proving the theorem, let us introduce some notations.

Given an interval $J$, let $\partial J$ denote its boundary.
For each $k,\ell\geq1$, the functions $f_k$ and $f_k^\ell$
will be denoted by $f(k)$ and $f(k,\ell)$ respectively.
In the same way we denote $g_k$ and $g_k^\ell$ by $g(k)$ and $g(k,\ell)$ respectively.

Fix $n\geq1$. For each $k\geq1$ let $\cc^k_n(\F)=\{x\in J^\F: f(n,\ell)(x)=
c_{n+\ell}^\F\; \textrm{for some}\; 0\leq\ell\leq k-1\}$. Note that
$\cc_n(\F)=\bigcup_{k\geq 1}\cc^k_n(\F)$.
Since $\cc_n^{k+1}(\F)=\cc_n^k(\F)\cup f(n,k)^{-1}(c_{n+k}^\F)$,
we have that $\cc_n^k(\F)\subset\cc_n^{k+1}(\F)$. Let
$\pp_n^k(\F)=\{J_n^\F(\ell)\subset J^\F:\; \partial J_n^\F(\ell)
\subset\cc_n^{k}(\F)\cup\{a^\F,b^\F\}
\; \textrm{and}\; 1\leq\ell\leq\theta^k_n(\F)\}$, 
where the increasing order of the index of 
$J_n^\F(\ell)$ is the same order as the intervals are placed in $J^\F$
and $\theta^k_n(\F):=\#\pp_n^{k}(\F)$.
By definition, $J_n^\F(\ell)$ is a maximal
monotonicity closed interval of $f(n,k)$ for all $1\leq\ell\leq \theta^k_n(\F)\}$.

We denote $f(k)\restriction_{[a^\F,c_k^\F]}$ and 
$f(k)\restriction_{[c_k^\F,b^\F]}$
by $f_-(k)$ and $f_+(k)$ respectively. Furthermore, given a sequence
$j=j_1j_2\ldots j_k \in\{-, +\}^k$ put
\[ 
f_j(n,k):=f_{j_k}(n+k-1)\circ\cdots\circ f_{j_2}(n+1)\circ f_{j_1}(n).
\]
Given $x\in\cc_n^k(\F)\backslash\{c_n^\F\}$ we 
consider $j=j_1j_2\ldots j_\ell \in\{-, +\}^\ell$, with $1\leq\ell\leq k-1$,
the minimal sequence such that $f_j(n,\ell)(x)=c_{n+\ell}^\F$.
We denote $x$ by $x_j^\F(n)$. Consider the set 
\[
\aaa_n^k(\F):=\biggl\{j=j_1\ldots j_\ell\in\{-, +\}^\ell:  
\begin{array}{c}
\exists x\in\cc_n^k(\F)\textrm{ such that }\\
x=x^\F_j(n)\textrm{ for some }0\leq\ell\leq k-1
\end{array}
\biggr\},
\]
where $c_n^\F=x_j^\F(n)$ when $j=\emptyset$ ($\ell=0$). Thus
\[
\cc^k_n(\F)=\left\{x_j^\F(n):\; j\in\{-,+\}^\ell\; \textrm{with}\; 0\leq\ell\leq k-1\right\}.
\]

Finally, note that for each $J_n^\F(\ell)\in\pp_{n}^k(\F)$ there is a unique sequence
$j=j_1\ldots j_k\in\{-,+\}^k$ such that 
$f_{j}(n,k)$ is strictly monotone on $J_n^\F(\ell)$. So 
\[
\partial J_n^\F(\ell)\in\left\{a^\F, b^\F, c_n^\F, x^\F_{j_1\ldots j_m}(n)\right\},
\]
for $1\leq m\leq k-1$. 

In the same way we can define $\cc_{n}^k(\G)$, $\pp_{n}^k(\G)$ and $\aaa_n^k(\G)$.

So the main idea in the proof is to show that for each $k\geq1$ there exists a strictly 
increasing bijection map $h_n^k:\cc_n^k(\F)\to\cc_n^k(\G)$ (see Figure \ref{fig1}) such that:
\begin{enumerate}[$\circ$]
\item $h_{n+1}^k\circ f(n)=g(n)\circ h_n^k$ on $\cc_n^k(\F)\backslash\{c_n^\F\}$,
\item $h_n^k\restriction_{\cc_n^{k-1}(\F)}=h_n^{k-1}$.
\end{enumerate}

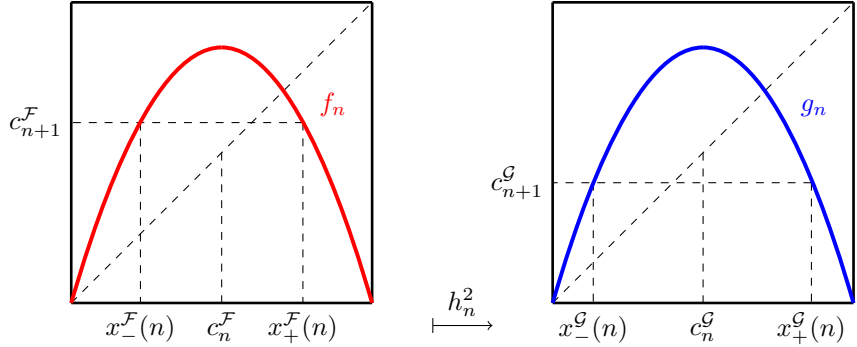
\begin{figure}[H]\centering
\begin{tikzpicture}[scale=2]
      \draw[line width=1pt] (1,-1) -- (-1,-1); 
			\draw[line width=1pt] (1,1) -- (1,-1); 
			\draw[line width=1pt] (1,1) -- (-1,1); 
			\draw[line width=1pt] (-1,1) -- (-1,-1); 
			\draw[|->] (1.4,-1.15) -- (1.8,-1.15) ;
			\fill[black] (1.6,-1.15) circle (0.mm) node[above] {$h_n^2$};
			\draw[dashed] (0.54,0.2) -- (-1,0.2) node[left]{$c_{n+1}^\F$};
      \draw[dashed] (0,0) -- (0,-1) node[below] {$c_n^\F$};
			\draw[dashed] (-1,-1) -- (1,1);
			\draw[dashed] (-0.54,0.2) -- (-0.54,-1) node[below]{$x_-^\F(n)$};
			\draw[dashed] (0.54,0.2) -- (0.54,-1) node[below]{$x_+^\F(n)$};
  \draw[color=red,line width=1.5pt] (0,0.7) parabola (-1,-1) ;
	\draw[color=red,line width=1.5pt] (0,0.7) parabola (1,-1) ;
	\fill[red,thick] (0.74,0.18) circle (0.mm) node[above] {$f_n$};
	
      \draw[line width=1pt] (4.2,-1) -- (2.2,-1); 
			\draw[line width=1pt] (4.2,1) -- (4.2,-1); 
			\draw[line width=1pt] (4.2,1) -- (2.2,1); 
			\draw[line width=1pt] (2.2,1) -- (2.2,-1); 
			\draw[dashed] (2.2,-1) -- (4.2,1);
			\draw[dashed] (3.92,-0.2) -- (2.2,-0.2) node[left]{$c_{n+1}^\G$};
      \draw[dashed] (3.2,0) -- (3.2,-1) node[below] {$c_n^\G$};
			\draw[dashed] (2.47,-0.2) -- (2.47,-1) node[below]{$x_-^\G(n)$};
			\draw[dashed] (3.92,-0.2) -- (3.92,-1) node[below]{$x_+^\G(n)$};  
			\draw[color=blue,line width=1.5pt] (3.2,0.7) parabola (2.2,-1) ;
	    \draw[color=blue,line width=1.5pt] (3.2,0.7) parabola (4.2,-1) ;
  \fill[blue,thick] (3.94,0.18) circle (0.mm) node[above] {$g_n$};
\end{tikzpicture}
\caption{Construction of $h_n^2$.}
\label{fig1}
\end{figure}

We proceed by induction on $k$. 

\begin{proof}[Proof of Theorem \ref{thm1}.]
Fix $n\geq1$. It is clear that we can define $h_n^1:\cc_n^1(\F)\to\cc_n^1(\G)$ 
as $h_n^1(c_n^\F):=c_n^\G$. Suppose by induction that there exists
$h_n^k:\cc_n^k(\F)\to\cc_n^k(\G)$ strictly increasing bijection such that 
$h_{n+1}^k\circ f(n)=g(n)\circ h_n^k$ on $\cc_n^k(\F)\backslash\{c_n^\F\}$
and $h_n^k\restriction_{\cc_n^{k-1}(\F)}=h_n^{k-1}$. This implies that
$\#\pp_{n}^k(\F)=\#\pp_{n}^k(\G)$ and $\aaa_n^k(\F)=\aaa_n^k(\G)$.

Let $J_n^\F(i)\in\pp_{n}^k(\F)$ and
$j=j_1\ldots j_k\in\{-,+\}^k$ such that
\[
J_n^\F(i)=\left[ x_{j_1\ldots j_m}^\F(n), x_{j_1\ldots j_\ell}^\F(n) \right],
\]
with $1\leq m\neq\ell\leq k-1$. By the manner that we order the intervals 
on $\pp_{n}^k(\F)$ and $\pp_{n}^k(\G)$ we have 
\[
h_n^k(x_{j_1\ldots j_{m}}^\F(n))=x_{j_1\ldots j_{m}}^\G(n)  \; \textrm{and} \;
h_n^k(x_{j_1\ldots j_{\ell}}^\F(n))=x_{j_1\ldots j_{\ell}}^\G(n),
\]
where 
\[
J_n^\G(i)=\left[ x^\G_{j_1\ldots j_m}(n), x_{j_1\ldots j_\ell}^\G(n) \right].
\]
Furthermore,
\begin{eqnarray*}
f_j(n,k)(J_n^\F(i))&=&[f_{j_k}(n+k-1)\circ\cdots\circ f_{j_{m+1}}(n+m)(c_{n+m}^\F), \\
                   & & \qquad \qquad \qquad \qquad 
										   f_{j_k}(n+k-1)\circ\cdots\circ f_{j_{\ell+1}}(n+\ell)(c_{n+\ell}^\F)]
\end{eqnarray*}
and
\begin{eqnarray*}
g_j(n,k)(J_n^\G(i))&=&[g_{j_k}(n+k-1)\circ\cdots\circ g_{j_{m+1}}(n+m)(c_{n+m}^\G), \\
                   & & \qquad \qquad \qquad \qquad 
										   g_{j_k}(n+k-1)\circ\cdots\circ g_{j_{\ell+1}}(n+\ell)(c_{n+\ell}^\G)].
\end{eqnarray*}

Since $\V(\F)=\V(\G)$, 
\[
f_j(n,k)(J_n^\F(i))\cap\{c_{n+k}^\F\}\neq\emptyset
\]
if, and only if
\[
g_j(n,k)(J_n^\G(i))\cap\{c_{n+k}^\G\}\neq\emptyset.
\]

If $f_j(n,k)(J_n^\F(i))\cap\{c_{n+k}^\F\}=\emptyset$, then
$J_n^\F(i)\in\pp_n^{k+1}(\F)$.
On the other hand, if $\partial\left(f_j(n,k)(J_n^\F(i))\right)\cap\{c_{n+k}^\F\}\neq\emptyset$
we get $J_n^\F(i)\in\pp_n^{k+1}(\F)$.
In any case, we also have $J_n^k(\G)\in\pp_n^{k+1}(\G)$.
Thus, $x_{j_1\ldots j_m}^\F(n)$ and $x_{j_1\ldots j_\ell}^\F(n)$ are the only points
of $\cc_n^{k+1}(\F)$ on $J_n^\F(i)$. The same thing holds on $J_n^\G(i)$. 

Now, if $\textrm{int}\left(f_j(n,k)(J_n^\F(i))\right)\cap\{c_{n+k}^\F\}\neq\emptyset$,
then there exists a unique 
\[
x_j^\F(n):=f_j(n,k)^{-1}(c_{n+k}^\F)\in\textrm{int}(J_n^\F(i)). 
\]
Hence, there also exists a unique
\[
x_j^\G(n):=g_j(n,k)^{-1}(c_{n+k}^\G)\in\textrm{int}(J_n^\G(i)).
\]
We define
\[ 
h_n^{k+1}(x_j^\F(n)):=x_j^\G(n),
\]
where $x_j^\F(n)\in\cc_n^{k+1}(\F)\backslash\cc_{n}^k(\F)$
and $x_j^\G(n)\in\cc_n^{k+1}(\G)\backslash\cc_{n}^k(\G)$.
Moreover, if $x\in\cc_n^k(\F)$, then we put $h_n^{k+1}(x)=h_n^k(x)$.

The cases where  $J_n^\F(i)=[a^\F, x_{j_1\ldots j_\ell}^\F(n)]$,
$J_n^\F(i)=[x^\F_{j_1\ldots j_\ell}(n), b^\F]$,
or $J_n^\F(i)=[x^\F_{j_1\ldots j_\ell}(n), c_n^\F]$, with $1\leq\ell\leq k-1$, are similar. Thus
we have that $h_n^{k+1}:\cc_n^{k+1}(\F)\to\cc_n^{k+1}(\G)$ is
a strictly increasing bijection such that:
\begin{enumerate}[$\circ$]
\item $h_{n+1}^{k+1}\circ f(n)=g(n)\circ h_n^{k+1}$ on $\cc_n^{k+1}(\F)\backslash\{c_n^\F\}$, and
\item $h_n^{k+1}\restriction_{\cc_n^k(\F)}=h_n^k$.
\end{enumerate}

Therefore, for all $n\geq1$ we can define
\[
\begin{array}{cccl}
h_n \ : & \! \cc_n(\F) & \! \longrightarrow  & \! \cc_n(\G) \\
        & \! w & \! \longmapsto      & \! h_n(w)=h_n^k(w),
\end{array}
\]
where $k$ is such that $w\in\cc_n^k(\F)$. Note that
$h_{n+1}\circ f(n)=g(n)\circ h_n$ on $\cc_n(\F)\backslash\{c_n^\F\}$.
This concludes the proof that $(J^\F,\F)$ is combinatorially equivalent
to $(J^\mathcal{G},\mathcal{G})$.
\end{proof}


\section{Topological conjugacy for UNDS}

Once that combinatorial equivalence has been established by the equivalence 
between kneading sequences, our goal now is to extend each bijection $h_n$
to the entire interval $J^\F$ to get a topological conjugacy. 

As a consequence of Theorem \ref{thm1}, we have the following

\begin{proposition}\label{prop2}
Let $(J^\F,\F)$ and $(J^\G,\G)$ be two UNDS with
kneading sequences $\V(\F)$ and $\V(\mathcal{G})$. 
Assume that $\cc_n(\mathcal{G})$ is dense in $J^\G$ for all $n\geq1$.
If $\V(\F)=\V(\G)$, then $(J^\F,\F)$ and $(J^\G,\mathcal{G})$
are combinatorially equivalent and $h_n:\cc_n(\F)\to\cc_n(\G)$ can be extended
continuously to $J^\F$ as an order preserving continuous surjective map 
such that $h_{n+1}\circ f_n=g_n\circ h_n$ for each $n\geq1$.
\end{proposition}

\begin{proof}
We only need to prove that each $h_n$ can be extended to $J^\F$. 
Fix $n\geq1$ and let $h_n:\cc_n(\F)\to\cc_n(\G)$.
We claim that we can extend continuously $h_n$ to $\overline{\cc_n(\F)}$.
In fact, take $w\in\overline{\cc_n(\F)}\backslash\cc_n(\F)$. Suppose that there
exist $w_k^j\in\cc_n(\F)$, $j=1,2$, such that $w_k^1\uparrow w$
and $w_k^2\downarrow w$. The cases where we have only $w_k\uparrow w$ or 
$w_k\downarrow w$, with $w_k\in\cc_n(\F)$, are similar.
Since $h_n$ is strictly increasing there are unique
\[
h_n^1(w):=\displaystyle\lim_k h_n(w_k^1) \ \ \textrm{and} \ \
h_n^2(w):=\displaystyle\lim_k h_n(w_k^2). 
\]
Note that $h_n^j(w)$ does not depend on the sequence $w^j_k$ converging to $w$, and 
$h_n^1(w)\leq h_n^2(w)$. Since $\cc_n(\G)$ is dense in $J^\G$ and $h_n$ is a strictly monotone map
we get $h_n^1(w)=h_n^2(w)$. 
Therefore, for $w\in\overline{\cc_n(\F)}\backslash\cc_n(\F)$ we can define
\[
h_n(w):=\displaystyle\lim_k h_n(w_k),
\]
where $w_k\in\cc_n(\F)$ is some (any) sequence such that $w_k\to w$.
Whence $h_n:\overline{\cc_n(\F)}\to J^\G$ is a strictly increasing continuous  
surjective map. 

Now, let $J=(\alpha, \beta)$ be a connected component
of $J^\F\backslash\overline{\cc_n(\F)}$. Once $h_n(\alpha)=h_n(\beta)$,
we can extend $h_n$ to $J$ as $h_n(w):=h_n(\alpha)$ for all $w\in J$.
Observe that $h_n(a^\F)=a^\F$ and $h_n(b^\F)=b^\F$.

\medskip
\noindent
{\sc Claim:} $h_{n+1}\circ f_n=g_n\circ h_n$ on $J^\F$ for all $n\geq1$.

\medskip
\noindent
{\em Proof of the claim.} From Theorem \ref{thm1}, we have 
$h_{n+1}\circ f_n=g_n\circ h_n$ on $\cc_n(\F)\backslash\{c_n^\F\}$.
Take $w\in J^\F\backslash\cc_n(\F)$ so that there is a sequence 
$(w_k)_{k\geq1}\in\cc_n(\F)\backslash\{c_n^\F\}$ satisfying $\displaystyle\lim_k w_k=w$. Thus, 
\[
(h_{n+1}\circ f_n)(w)=\displaystyle\lim_k (h_{n+1}\circ f_n)(w_k)
=\displaystyle\lim_k (g_n\circ h_n)(w_k)=(g_n\circ h_n)(w).
\]
Let $J=(\alpha,\beta)$ be a connected component of $J^\F\backslash\overline{\cc_n(\F)}$.
Observe that $f_n(J)$ also is a connected component of $J^\F\backslash\overline{\cc_{n+1}(\F)}$.
Without loss of generality we may assume that 
$\alpha\in\overline{\cc_n(\F)\backslash\{c_n^\F\}}$ or $\alpha\in\{a^\F, b^\F\}$,
%
%
and so $h_{n+1}(f_n(\alpha))=g_n(h_n(\alpha))$. Thus
\[ 
h_{n+1}(f_n(w))=h_{n+1}(f_n(\alpha))=g_n(h_n(\alpha))=g_n(h_n(w)),
\]
%
%
for all $w\in J$. It remains to prove that $(h_{n+1}\circ f_n)(c_n^\F)=(g_n\circ h_n)(c_n^\F)$.
If $c_n^\F\in\overline{\cc_n(\F)\backslash\{c_n^\F\}}$, then equality 
follows from the first part. Otherwise, there exists $\alpha\in J^\F$ so that
$(\alpha, c_n^\F)\subset J^\F\backslash\overline{\cc_n(\F)}$ and 
$\alpha\in\overline{\cc_n(\F)\backslash\{c_n^\F\}}$ or $\alpha\in\{a^\F, b^\F\}$,
then equality follows from the second one. This proves the claim.
%
%

\medskip
The proof of the proposition is finished.
\end{proof}

Notice that if we also assume that $\cc_n(\F)$ is dense in $J^\F$ for all 
$n\geq1$ on the proposition above, then $h_n:J^\F\to J^\G$ is a
homeomorphism for all $n\geq1$.  

\medskip
\noindent 
{\sc Limit Property:} Let $(J^\F, \F)$ be a UNDS. We say that $(J^\F,\F)$ satisfies the 
{\em limit property} if there is a unimodal map $f:J^\F\to J^\F$ 
so that the turning point $c^f$ of $f$ is not periodic, and such that: 
\begin{enumerate}[$\circ$]
\item $f_n$ converges uniformly to $f$;
\item $\cc_n(\F)$ is dense in $J^\F$, $\forall n\geq1$;
\item $\cc(f)$ is dense in $J^\F$.
\end{enumerate}

\medskip

Now we are ready to prove Theorem \ref{thmA}.

\begin{proof}[Proof of Theorem \ref{thmA}.]
Because of Proposition \ref{prop1} we only need to prove the reverse 
implication. From Proposition \ref{prop2}, it remains to prove that the 
families $(h_n)$ and $(h_n^{-1})$ are equicontinuous. By symmetry, it is enough 
to prove that $(h_n)$ is equicontinuous. We will use the same notation 
to the one used in the proof of Theorem \ref{thm1}. Let $\varepsilon>0$.
Since $(J^\G,\G)$ has the limit property, there exist a unimodal map 
$g:J^\G\to J^\G$ and $k\geq1$ so that $\cc^k(g)$ is 
$\tfrac{\varepsilon}{8}$--dense in $J^\G$. Since $c^g$ is not periodic, for each $m=1, \ldots, k-1$
we have that if $x_\alpha(g)$ and $x_\beta(g)$ are two consecutive elements 
of $\cc^m(g)$ on $J^\G$, then either $c^g\notin[g^m(x_\alpha(g)),g^m(x_\beta(g))]$
or $c^g\in(g^m(x_\alpha(g)),g^m(x_\beta(g)))$. 
Therefore, using that $g(n+m-1)\circ\cdots\circ g(n)$ converges uniformly 
to $g^m$ for each $m\geq1$
, and proceeding as \cite[Lemma 3.1]{EA}, 
there is $N_1\geq1$ (see Figure \ref{fig2}) such that $\aaa_n^k(\G)=\aaa^k(g)$ and
\[
d(x_j(g),x_j^\G(n))<\tfrac{\varepsilon}{8} \; \text{for each}\; j\in\aaa^k(g) \;\text{and}\; n\geq N_1.
\]
This implies that $\cc_n^k(\G)$ is $\tfrac{\varepsilon}{4}$--dense in $J^\G$ for all $n\geq N_1$.

\begin{figure}[H]\centering
\begin{tikzpicture}[scale=1.4]
      \draw[line width=1pt] (2,-2) -- (-2,-2); 
			\draw[line width=1pt] (2,2) -- (2,-2); 
			\draw[line width=1pt] (2,2) -- (-2,2); 
			\draw[line width=1pt] (-2,2) -- (-2,-2); 
			\draw[dashed] (-1.1,-0.8) -- (-1.1,-2); 
			\fill[black] (-1.15,-2) circle (0.mm) node[below]{$a$};
			\draw[dashed] (-1,-0.7) -- (-1,-2);
			\fill[black] (-1,-1.95) circle (0.mm) node[below]{$b$};
			\draw[dashed] (0.4,1.2) -- (0.4,-2); 
			\fill[black] (0.35,-2) circle (0.mm) node[below]{$c$};
      \draw[dashed] (0.5,1.2) -- (0.5,-2);
			\fill[black] (0.5,-1.95) circle (0.mm) node[below]{$d$};
			\draw[dashed] (-2,0.2) -- (-0.6,0.2);
			\draw[dashed] (-2,0.3) -- (-0.45,0.3);
			\draw[dashed] (-0.58,0.2) -- (-0.58,-2);
			\fill[black] (-0.57,-2) circle (0.mm) node[below]{$e$};
			\draw[dashed] (-0.45,0.3) -- (-0.45,-2);
			\fill[black] (-0.45,-1.95) circle (0.mm) node[below]{$f$};
			\coordinate (A) at (-1.1,-0.8);
      \coordinate (B) at (0.4,1.2);
	    \draw[color=blue,line width=1.5pt]  (A) to [bend left=15] (B);
			\coordinate (C) at (-1,-0.7);
      \coordinate (D) at (0.5,1.2);
	    \draw[color=red,line width=1.5pt]  (C) to [bend left=15] (D);
			\fill[black] (-2,0.1) circle (0.mm) node[left]{$c^g$};
			\fill[black] (-2,0.4) circle (0.mm) node[left]{$c_{n+m}^\G$};
			\fill[black] (0.7,1.8) circle (0.mm) node[right]{$a=x_\alpha(g)$};
			\fill[black] (0.7,1.4) circle (0.mm) node[right]{$b=x_\alpha^\G(n)$};
			\fill[black] (0.7,1) circle (0.mm) node[right]{$c=x_\beta(g)$};
			\fill[black] (0.7,0.6) circle (0.mm) node[right]{$d=x_\beta^\G(n)$};
			\fill[black] (0.7,0.2) circle (0.mm) node[right]{$e=x_j(g)$};
			\fill[black] (0.7,-0.2) circle (0.mm) node[right]{$f=x_j^\G(n)$};
			\fill[black] (0.6,-0.6) circle (0.mm) node[right]{$d(e,f)<\tfrac{\varepsilon}{8}$};
			\fill[blue] (-1.8,1.8) circle (0.mm) node[right]{$g^m$}; 
			\fill[black] (-1.5,1.8) circle (0.mm) node[right]{,\;$1\leq m\leq k-1$};
			\fill[red] (-1.8,1.4) circle (0.mm) node[right]{$g(n,m)$};
			\fill[black] (-1,1.4) circle (0.mm) node[right]{,\;$n\geq N_1$};
 
\end{tikzpicture}
\caption{Case $c^g\in(g^m(x_\alpha(g)),g^m(x_\beta(g)))$. In the other
case, $x_j(g)$ and $x_j^\G(n)$ do not exist.}
\label{fig2}
\end{figure}
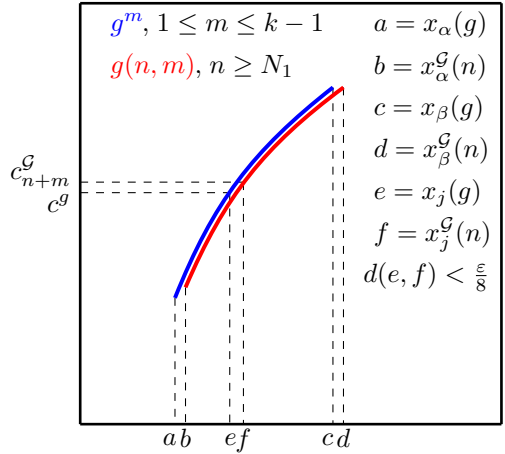

\medskip
\noindent
{\sc Claim 1:} For any $n\geq N_1$, let $x_j^\G(n)$ and $x_{\ell}^\G(n)$ be two consecutive elements
of $\cc_n^k(\G)$ on $J^\G$. Then $d(x_j^\G(n), x_{\ell}^\G(n))<\tfrac{\varepsilon}{2}$.

\medskip
\noindent
{\em Proof of the claim.} Suppose by contradiction this claim is false.
Let $x$ be the midpoint of the interval $[x_j^\G(n), x_\ell^\G(n)]$.
Hence $d(x, x_j^\G(n))\geq\tfrac{\varepsilon}{4}$ and $d(x, x_{\ell}^\G(n))\geq\tfrac{\varepsilon}{4}$.
This contradicts the $\tfrac{\varepsilon}{4}$--density of $\cc_n^k(\G)$ 
and finishes the proof of the claim.   

\medskip
By the limit property of $(J^\F,\F)$ and $\V(\F)=\V(\G)$, there exists $N_2\geq1$
such that $\aaa_n^k(\F)=\aaa^k(f)=\aaa^k(g)$ and
\[
d(x_j(f),x_j^\F(n))<\tfrac{M}{4} \; \text{for each}\; j\in\aaa^k(f) \;\text{and}\; n\geq N_2,
\]
where
\[
M=\displaystyle\min_{\substack{j,\ell\in\aaa^k(f)\\j\neq\ell}}\left\{d(x_j(f), x_{\ell}(f)),
d(x_j(f), a^\F), d(x_j(f), b^\F)\right\}.
\]

\medskip
\noindent
{\sc Claim 2:} For any $n\geq N_2$, let $x_j^\F(n)$ and $x_{\ell}^\F(n)$ be two consecutive elements
of $\cc_n^k(\F)$ on $J^\F$. Then $d(x_j^\F(n), x_{\ell}^\F(n))\geq\tfrac{M}{2}$.

\medskip
\noindent
{\em Proof of the claim.} Suppose that $\alpha,\beta\in\aaa^k(f)$ 
are such that $x_\alpha(f)$ and $x_\beta(f)$ are consecutive on $J^\F$ and 
$d(x_\alpha(f), x_\beta(f))=M$. The other possibilities $d(x_\alpha(f), a^\F)=M$,  
and $d(x_\alpha(f), b^\F)=M$, are treated similarly. 
For any two consecutive elements $x_j^\F(n)$ and $x_{\ell}^\F(n)$ of $\cc_n^k(\F)$
we have that $d(x_j^\F(n), x_{\ell}^\F(n))\geq d(x_j(f), x_\ell(f))-2\tfrac{M}{4}$.
Hence $d(x_j^\F(n), x_{\ell}^\F(n))\geq d(x_\alpha(f), x_\beta(f))-2\tfrac{M}{4}=\tfrac{M}{2}$   
and the proof of the claim is finished.

\medskip
Take $0<\delta_1<\tfrac{M}{2}$ and $N=\max\{N_1, N_2\}$. Let $x,y\in J^\F$ 
such that $d(x, y)<\delta_1$ and $n\geq N$. Without loss of generality,
we may assume that $x\leq y$. We have two cases to consider:  
\begin{enumerate}[$\circ $]
\item There are $x_\alpha^\F(n)$ and $x_\beta^\F(n)$ two consecutive elements
of $\cc_n^k(\F)$ so that $x,y\in[x_\alpha^\F(n), x_\beta^\F(n)]$. By construction
of $h_n$, $x_\alpha^\G(n)$ and $x_\beta^\G(n)$ are two consecutive elements
of $\cc_n^k(\G)$ and $h_n(x),h_n(y)\in[x_\alpha^\G(n), x_\beta^\G(n)]$.
Claim $1$ implies that $d(h_n(x), h_n(y))<\tfrac{\varepsilon}{2}$.
\item There are $x_\alpha^\F(n)$, $x_\beta^\F(n)$ and $x_\gamma^\F(n)$ three
consecutive elements of $\cc_n^k(\F)$ so that $x\in[x_\alpha^\F(n), x_\beta^\F(n)]$
and $y\in[x_\beta^\F(n), x_\gamma^\F(n)]$ (Claim $2$). Again by construction of $h_n$,
$x_\alpha^\G(n)$, $x_\beta^\G(n)$ and $x_\gamma^\G(n)$ are three consecutive elements
of $\cc_n^k(\G)$ and $h_n(x)\in[x_\alpha^\G(n), x_\beta^\G(n)]$ and
$h_n(y)\in[x_\beta^\G(n), x_\gamma^\G(n)]$. Consequently $d(h_n(x), h_n(y))<\varepsilon$.
\end{enumerate} 

Now, by continuity there exists $\delta_2>0$ such that $x,y\in J^\F$ with
$d(x,y)<\delta_2$ implies $d(h_n(x),h_n(y))<\varepsilon$ for each $n=1, 2, \ldots, N-1$.
Hence taking $\delta=\min\{\delta_1,\delta_2\}$ the theorem follows.
\end{proof}

Notice that the same proof above can be used to prove a variant of Theorem \ref{thmA}: 
Let $(J^\F,\F)$ and $(J^\G,\G)$ be two unimodal nonautonomous
discrete dynamical systems, and assume that
$(J^\G,\G)$ satisfies the limit property. If $(J^\F,\F)$ and $(J^\G,\G)$
have the same kneading sequences and there exists a unimodal map $f:J^\F\to J^\F$
such that the turning point of $f$ is not periodic, $\overline{\cc(f)}=J^\F$ and $f_n\to f$,
then they are topologically semi-conjugate.


\begin{example}
For each $n\geq1$, let $f_n,g_n,q_n:[0,1]\to[0,1]$ be three unimodal maps defined by
$$
f_n(x)= 
\left\{
\begin{array}{ll}
(\frac{2n+4}{4+n})x  &\textrm{, if } x\in[0, \frac{4+n}{2n+4}] \\
(\frac{2n+4}{n})(1-x)&\textrm{, if } x\in[\frac{4+n}{2n+4},1]
\end{array}
\right.,
$$
$$
g_n(x)= \left\{
\begin{array}{ll}
(\frac{4n+4}{3n-1})x   &\textrm{, if } x\in[0, \frac{3n-1}{4n+4}] \\
(\frac{4n+4}{n+5})(1-x)&\textrm{, if } x\in[\frac{3n-1}{4n+4}, 1] 
\end{array}
\right.,
$$
$$ 
q_n(x)= \left\{
\begin{array}{ll}
2(\frac{n+4}{n+5})x   &\textrm{, if } x\in[0, \frac{1}{2}] \\
2(\frac{n+4}{n+5})(1-x)&\textrm{, if } x\in[\frac{1}{2}, 1]. 
\end{array}
\right.
$$
Observe that $f_n, q_n$ converges uniformly to $f$ and $g_n$ converges uniformly to
$g$, where
$$
f(x)= \left\{
\begin{array}{ll}
2x   &\textrm{, if } x\in[0, \frac{1}{2}] \\
2(1-x)&\textrm{, if } x\in[\frac{1}{2}, 1] 
\end{array}
\right.,
\ \ g(x)= \left\{
\begin{array}{ll}
\frac{4}{3}x   &\textrm{, if } x\in[0, \frac{3}{4}] \\
4(1-x)&\textrm{, if } x\in[\frac{3}{4}, 1]. 
\end{array}
\right.
$$
It is easy to see that $([0,1], (f_n))$, $([0,1], (g_n))$, and 
$([0,1], (q_n))$ satisfy the limit property. From Theorem \ref{thmA},
it follows that $([0,1], (f_n))$ is topologically conjugate to $([0,1], (g_n))$.
On the other hand, $([0,1], (f_n))$ is not topologically conjugate to $([0,1], (q_n))$,
since $q_n$ is not surjective for any $n\geq1$. More generally, we can construct a 
family $\mathscr{M}$ of UNDS satisfying the limit property containing
conjugated and non-conjugated UNDS. For that, let $\mathscr{U}$ be the set of 
unimodal maps $f:[0,1]\to[0,1]$ so that $|f'|\geq\theta(f)>1$. 
Let $\mathscr{M}:=\{([0,1], \F): f_n\to f \text{ with } f,f_n\in\mathscr{U} 
\text{ and } f \text{ has no periodic turning point}\}$.
Observe that the three UNDS constructed above belong to $\mathscr{M}$. 
\end{example}

We finish this section remarking that many concepts used in classical dynamics 
such as periodicity, recurrence and wandering domains have generalizations 
for nonautonomous systems even though none work easily, while other concepts 
such as entropy can be generalized. As the notion of attracting periodic points 
and wandering intervals are the main obstructions to the density of critical set 
(preimages of turning points) on the autonomous setting, it is natural we ask:

\medskip
\noindent
{\bf Problem 2:} Let $(J^\F,\mathcal{F})$ be a unimodal nonautonomous 
discrete dynamical systems. Under which conditions we have that
$\cc_n(\F)$ is dense in $J^\F$ for all $n\geq1$?

\medskip


\appendix

\section{The multimodal case}

In this appendix, we work with NDS where each map $f_n:J^\F\to J^\F$ is a multimodal map.
Let $(J^\mathcal{F}, \mathcal{F})$ be the NDS defined as follows:
let $J^\F=[a^\F, b^\F]$ be an interval, and 
$f_n:J^\F\to J^\F$ be a  piecewise monotone continuous map such that
either $f_n(a^\F)=a^\F$ for each $n\geq1$ or $f_n(a^\F)=b^\F$ for each $n\geq1$
and either $f_n(b^\F)=a^\F$ for each $n\geq1$ or $f_n(b^\F)=b^\F$ for each $n\geq1$. 
Furthermore, there exists $\ell\geq1$ so that for each $n\geq1$ 
there are $a^\F<c_n^\F(1)<\cdots<c_n^\F(\ell)<b^\F$ 
such that the intervals $I_n^\F(1)=\left[a^\F, c_n^\F(1)\right),
I_n^\F(2)=\left(c_n^\F(1), c_n^\F(2)\right), \ldots, I_n^\F(\ell+1)=\left(c_n^\F(\ell), b^\F\right]$ 
are the largest intervals in which $f_n$ is strictly monotone.
The points $\{c_n^\F(1), \ldots, c_n^\F(\ell):\; n\geq1\}$ are the {\it turning points} of the NDS.
We call the NDS defined above $\ell$-{\it modal nonautonomous discrete dynamical system} (short $\ell$-MNDS)

Consider the alphabet $\mathcal{A}_\F=\{I_n^\F(1), c_n^\F(1), I_n^\F(2), 
\ldots, c_n^\F(\ell), I_n^\F(\ell+1):\; n\geq1\}$.

\medskip
\noindent
{\sc Address of a point:} Let $n\geq1$. The {\em address} of a point $x\in J^\F$ 
on the level $n$ is the letter $i_{\F,n}(x)\in\mathcal A_\F$ defined by
$$
i_{\F,n}(x)= \left\{
\begin{array}{ll}
I_n^\F(j)    & \textrm{, if }  x\in I_n^\F(j)     \\
c_n^\F(j)  & \textrm{, if }  x=c_n^\F(j). 
\end{array}
\right.
$$

\medskip
\noindent
{\sc Itinerary of a point:} Let $n\geq1$. The {\em itinerary} of a point $x\in J^\F$
on the level $n$ is the sequence $I_{\mathcal{F},n}(x)\in\mathcal A_\F^{\{0,1,2,\ldots\}}$ defined by 
\[
I_{\mathcal{F},n}(x)=(i_{\mathcal{F},n}(x),i_{\mathcal{F},n+1}(f_n^1(x)),\ldots,
i_{\F,n+\ell}(f_n^\ell(x)),\ldots).
\]

\medskip
\noindent
{\sc Kneading Sequences:} The {\em kneading sequences} of $(J^\F,\F)$ 
are the sequences $\V^\F(j)=\{\V_n^\F(j)\}_{n\geq1}$, 
where $\V_n^\F(j):=I_{\F,n}(c_n^\F(j))$ and $j=1, \ldots, \ell$.

\medskip 
\begin{definition}\label{defMNDS}
Let $(J^\F,\F)$ and $(J^\G,\G)$ be two $\ell$-MNDS. We say that
\begin{enumerate}[\rm (1)] 
\item {\em $(J^\F,\F)$ and $(J^\G,\G)$ have the same kneading sequences} if 
\[
\V^\F(j)=\V^\G(j),\ \textrm{for each } j=1, \ldots, \ell.
\]
\item $(J^\F,\F)$ is {\em monotonically equivalent} 
to $(J^\G,\G)$ if for all $n\geq1$ and each $j=1, \ldots, \ell+1$:
\begin{enumerate}[$\circ$]
\item either $f_n\restriction_{I_n^\F(j)}$ and $g_n\restriction_{I_n^\G(j)}$ are
strictly increasing;
\item or $f_n\restriction_{I_n^\F(j)}$ and $g_n\restriction_{I_n^\G(j)}$
are strictly decreasing.
\end{enumerate}
\end{enumerate}
\end{definition}

%
%
The limit property for $\ell$-MNDS has a little change: if $f$ is the 
$\ell$-modal map such that $f_n$ converges uniformly to $f$, then
$c^f(i)$ is not periodic and $c^f(i)\notin\{f^{-m}(c^f(j)): m\geq1\}$
for all $i,j=1,\ldots,\ell$. The other items remain the same.
With these more general definitions, we have the following. 

%
%
\begin{theorem}\label{thmA2}
Let $(J^\F,\F)$ and $(J^\G,\G)$ be two monotonically equivalent $\ell$-modal 
nonautonomous discrete dynamical systems, and assume that both satisfy the 
limit property. Then $(J^\F,\F)$ and $(J^\G,\G)$ are topologically 
conjugate if and only if they have the same kneading sequences.
\end{theorem}

The proof of Theorem \ref{thmA2} follows exactly the same ideas and arguments
as in the proof of Theorem \ref{thmA}, and we leave it to the reader.


\bibliography{NADS}
\bibliographystyle{alpha}

\end{document}